\documentclass[11pt]{article}
\usepackage[percent]{overpic}
\usepackage{amsmath}
\usepackage{amssymb}
\usepackage{amsthm}
\usepackage{array} 
\usepackage[english]{babel}
\usepackage{enumitem}
\usepackage{lmodern}
\usepackage{xfrac}  

\usepackage{graphicx,float,wrapfig, caption}
\usepackage[percent]{overpic}
\usepackage{pgf, tikz, tikz-cd} 
\graphicspath{ {pictures/}{tikz/} }
\usepackage[tmargin=0.75in,bmargin=0.75in,lmargin=0.75in,rmargin=0.75in]{geometry}
\usepackage{mathtools} 
\usetikzlibrary{arrows} 
\usepackage{extarrows}
\usetikzlibrary{positioning}
\usepackage{standalone} 

\usepackage{hyperref}
\hypersetup{
    unicode=false,          
    pdftoolbar=true,        
    pdfmenubar=true,        
    pdffitwindow=false,     
    pdfstartview={FitH},    
    pdftitle={},    
    pdfauthor={Leonid Monin},     
    pdfsubject={Subject},   
    pdfcreator={Leonid Monin},   
    pdfproducer={Producer}, 
    pdfkeywords={Linear series} {Spherical Varieties} {Newton Polyhedra}  {Newton Okounkov bodies} {Resultants} , 
    pdfnewwindow=true,      
    colorlinks=true,       
    linkcolor=blue,          
    citecolor=red,        
    filecolor=magenta,      
    urlcolor=cyan           
}

\makeatletter

\newcommand{\abstractin}[1]{%
  \otherlanguage{#1}%
  \item[\hskip\labelsep\scshape\abstractname.]%
}
\makeatother

 \renewcommand{\epsilon}{\varepsilon}

\newcommand{\Z}{{\mathbb Z}}
 
\newcommand{\R}{{\mathbb R}}
\newcommand{\C}{{\mathbb C}}

\newcommand{\Def}{\operatorname{def}}

\newcommand{\cA}{{\mathcal A}}
\newcommand{\cB}{{\mathcal B}}

 \newtheorem{Theorem}{Theorem}
 \newtheorem{Corollary}[Theorem]{Corollary}
  \newtheorem{Lemma}[Theorem]{Lemma}

\newtheorem{Proposition}[Theorem]{Proposition}

\theoremstyle{remark}
 \newtheorem{Remark}{Remark}
 
\newtheorem{Definition}{Definition}

\begin{document}

\title{Discrete Invariants of Generically Inconsistent Systems of Laurent Polynomials}
\author{Leonid Monin}

\date{\today}
\maketitle
\begin{abstract}
\abstractin{english}
Let $ \mathcal{A}_1, \ldots, \mathcal{A}_k $ be finite sets in $ \mathbb{Z}^n $ and let $ Y \subset (\mathbb{C}^*)^n $ be an algebraic variety defined by a system of equations
\[
f_1 = \ldots = f_k = 0,
\]
where $ f_1, \ldots, f_k $ are Laurent polynomials with supports in $\mathcal{A}_1, \ldots, \mathcal{A}_k$. Assuming that $ f_1, \ldots, f_k $ are sufficiently generic, the Newton polyhedron theory computes discrete invariants of $ Y $ in terms of the Newton polyhedra of $ f_1, \ldots, f_k $. It may appear that the generic system with fixed supports $ \mathcal{A}_1, \ldots, \mathcal{A}_k $ is inconsistent. In this paper, we compute discrete invariants of algebraic varieties defined by systems of equations which are generic {\em in the set of consistent system} with support in $ \mathcal{A}_1, \ldots, \mathcal{A}_k$ by reducing the question to the Newton polyhedra theory.  Unlike the classical situation, not only the Newton polyhedra of $f_1,\dots,f_k$, but also the supports $\cA_1,\dots,\cA_k$ themselves appear in the answers.

\end{abstract}

\section{Introduction.}
With a Laurent polynomial $f$  in $n$ variables one can associate its support $supp(f)\subset \Z^n$ which is the set of exponents of monomials having non-zero coefficient in $f$ and its Newton polyhedra $\Delta(f)\subset \R^n$ which is the convex hull of the support of $f$ in $\R^n$.  Consider an algebraic variety  $Y\subset (\C^*)^n$ defined by a system of equations 
\begin{equation}\label{eq}
f_1=\dots=f_k=0,
\end{equation}
 where $f_1,\dots, f_k$ are Laurent polynomials with the supports in finite sets $\cA_1, \ldots, \cA_k \subset \Z^n$. The Newton polyhedra theory computes invariants of $Y$ assuming that the system (\ref{eq}) is generic enough. That is, there exists a proper algebraic subset $\Sigma$ in the space $\Omega$ of $k$-tuples of Laurent polynomials $f_1,\ldots,f_k$  such that  the corresponding discrete invariant is constant in $\Omega \setminus \Sigma$ and could be computed in terms of polyhedra $\Delta_1,\ldots,\Delta_k$. One of the first examples of such result is the Bernstein-Kouchnirenko-Khovanskii theorem (see \cite{B}). 

\begin{Theorem}[BKK]
Let $f_1, \ldots, f_n$ be generic Laurent polynomials with supports in $\cA_1, \ldots, \cA_n$. Then all solutions  of the system $f_1 = \ldots = f_n = 0$ in $(\C^*)^n$  are non-degenerate and  the number of solutions is
$$
n! Vol(\Delta_1,\ldots,\Delta_n),
$$
where $\Delta_i$ is the convex hull of $\cA_i$ and $Vol$ is the mixed volume.
\end{Theorem}
For some of other examples see \cite{DKh}, \cite{Kh}, \cite{Kh2}. If $(f_1,\ldots,f_k)\in \Sigma$, the invariants of $Y$ depend not only on $\Delta_1,\ldots,\Delta_k$ and, in general, are much harder to compute.

In the case that $\cA_1, \ldots, \cA_k$ are such that the general system is inconsistent in $(\C^*)^n$ one can modify the question in the following way. {\emph What are discrete invariants of a zero set of generic consistent system with given supports? }The main result of this paper is Theorem~\ref{main} which reduces this question to the  Newton polyhedra theory. In this situation, the discrete invariants are computed in terms of supports themselves, not the Newton polyhedra. Some examples of applications of Theorem~\ref{main} are given in Section~5 (in particular we obtain a generalization of the BKK Theorem).
\subsection*{Acknowledgments}
The author would like to thank Askold Khovanskii for the proposing this problem and for his enthusiasm and support during the work as well as Bernd Schober and the referee for very useful comments on the earlier versions of the paper.

\section{Preliminary facts on the set of consistency.}

The material of this section is well-known (see for example \cite{GKZ}, \cite{St}, \cite{D'AS}). 
\subsection{Definition of the incidence variety and the set of consistency.}

Let $A=(\cA_1,\ldots,\cA_k)$ be a collection of $k$ finite subsets of the lattice  $\Z^n$. The space $\Omega_A$ of Laurent polynomials $f_1,\ldots f_k$ with supports in $\cA_1,\ldots,\cA_k$ is isomorphic to $(\C)^{|\cA_1|+\ldots+|\cA_k|}$, where $|\cA_i|$ is the number of points in $\cA_i$. 

\begin{Definition}
The {\it incidence variety} $\widetilde X_A\subset (\C^*)^n\times \Omega_A$ is defined as:
$$
\widetilde X_A=\{(p,(f_1,\ldots,f_k))\in(\C^*)^n\times \Omega_A | f_1(p)=\ldots=f_k(p)=0 \}.
$$
\end{Definition}

Let $\pi_1:(\C^*)^n \times \Omega_A \to (\C^*)^n$, $\pi_2:(\C^*)^n\times \Omega_A \to  \Omega_A$  be natural projections to the first and the second factors of the product. 

\begin{Definition}
The {\it set of consistency} $X_A\subset \Omega_A$ is the image of $\widetilde X_A$ under the projection $\pi_2$.
\end{Definition}

\begin{Theorem}
The  incidence variety $\widetilde X_A\subset (\C^*)^n\times \Omega_A$ is a smooth algebraic variety.
\end{Theorem}
\begin{proof}
Indeed, the projection $\pi_1$ restricted to $\widetilde X_A$:
$$
\pi_1:\widetilde X_A \to (\C^*)^n
$$
forms a vector bundle of rank  ${|\cA_1|+\ldots+|\cA_k|-k}$. That is because for a point $p\in (\C^*)^n$ the preimage $\pi_1^{-1}(p)\subset \widetilde X_A$ is given by $k$  independent linear equations on the coefficients of polynomials $f_1,\ldots,f_k$.
\end{proof}

We will say that a constructible subset $X$ of $\C^N$ is irreducible if for any two polynomials $f, g$ such that $fg|_X=0$ either $f|_X=0$ or $g|_X=0$. 

\begin{Corollary}
The set of consistency $X_A$ is an irreducible constructible subset of $\Omega_A$.
\end{Corollary}

\begin{proof}
 Since $X_A=\pi_2(\widetilde X_A)$ is the image of an irreducible algebraic variety $\widetilde X_A$ under the algebraic map $\pi_2$, it is constructible and irreducible.
\end{proof}

\subsection{Codimension of the set of consistency.}

 For a collection $B=(\cB_1,\ldots,\cB_\ell)$ of finite subsets of $\Z^n$ let $\cB=\cB_1+\ldots+\cB_\ell$ be the Minkowski sum of all subsets in the collection and let $L(B)$ be the linear subspace parallel to the minimal affine subspace containing $\cB$.

\begin{Definition}
The defect of a collection $B=(\cB_1,\ldots,\cB_\ell)$ of finite subsets of $\Z^n$ is given by
$$
\Def(\cB_1,\ldots,\cB_\ell)=\dim(L(B))-\ell.
$$
\end{Definition}

For a subset $J\subset \{1,\ldots, \ell\}$  let us define the collection $B_J = (\cB_i)_{i\in J}$. For the simplicity we denote the defect $\Def(B_J)$ by $\Def(J)$, and the linear space $L(B_J)$ by $L(J)$.

The following theorem provides a criterion for a system of Laurent polynomials with supports in $\cA_1,\ldots,\cA_k$ to be generically consistent.
\begin{Theorem}[Bernstein]
A system of generic equations $f_1=\ldots=f_k=0$ of Laurent polynomials with supports in $\cA_1,\ldots,\cA_k$ respectively has a common root if and only if for any $J\subset \{1,\ldots,k\}$  the defect $\Def(J)$ is nonnegative.
\end{Theorem}

According to the Bernstein theorem, if there exist subcollection of $A$ with negative defect, the codimension of the set of consistency is positive. We will call such collections $A$ {\it generically inconsistent}. The following theorem of Sturmfels determines the precise codimension of $X_A$. 
\begin{Theorem}[\cite{St}, Theorem 1.1]
Let $\cA_1,\ldots,\cA_k$ be such that the generic system with supports in $\cA_1,\ldots,\cA_k$ is inconsistent. Then the codimension of the set of consistency $X_A$ in  $\Omega_A$ is equal to the maximum of the numbers
$-\Def(J)$, where $J$ runs over all subsets of $\{1,\ldots,k\}$.
\end{Theorem}

\begin{Definition}
For  a collection $\cA_1,\ldots,\cA_k$ of finite subsets of $\Z^n$ we will denote by $d( \cA)$ the smallest defect of a subcollection of $A$:
$$
d(A) = \min \{  \mathrm{def}(J) \mid J \subset \{1, \ldots, k \} \, \}.
$$
We will say that a collection $A$ is  generically inconsistent if the minimal defect $d(A)$ is negative. 
\end{Definition}
\begin{Definition}
For a generically inconsistent collection $A$ will call a subcollection $J$ {\it essential} if $\Def(J)=d(A)$ and $\Def(I)>d(A)$ for any $I\subset J$. In other words, $J$ is the minimal by inclusion subcollection with the smallest defect. 
\end{Definition}

  This definition is related to the definition of an essential subcollection given in \cite{St}, but is different in general. Sturmfels was  interested in resultants, so his definition was adapted to the case $d(A)=-1$ in which both definitions coincide.

The essential subcollection is unique. For $d=-1$ this was shown in \cite{St} (Corollary 1.1), in Lemma~\ref{unique} we prove this statement for arbitrary $d<0$. In the case $d(A)=0$ we will call the empty subcollection to be the unique essential subcollection.

\begin{Remark}
In the case $d(A)=0$ the subcollections $J$ such that $\Def(J)=0$ and $\Def(I)>0$ for any {\em nonempty} $I\subset J$ are also playing important role (see  \cite{Kh2}). 
\end{Remark}

\section{The defect and essential subcollections.}

\subsection{Uniqueness of essential subcollection.}
Let $\cA_1,\ldots,\cA_k$ be finite subsets of the lattice  $\Z^n$. As before, for any $J\subset \{1,\ldots k\}$, let $L(J)$ be the vector subspace parallel to the minimal affine subspace containing the Minkowski sum $\cA_J=\sum \cA_i$ with $i\in J$.

Most of the results of this section are based on the obvious observation that  the dimension of vector subspaces of $\R^n$ is subadditive with respect to sums. That is for two vector subspaces $V,W\subset \R^n$ the following holds:
$$
 \dim(V + W)= \dim(V)+ \dim(W)- \dim(V\cap W)\leq \dim(V)+ \dim(W).
$$
The immediate corollary of the relation above is the subadditivity of defect  with respect to disjoint unions. More precisely, for disjoint $I, J \subset \{1,\ldots k\}$ the following is true:
\begin{equation}\label{sub}
\Def(I\cup J)=\Def(I)+\Def(J)-\dim(L(I)\cap L(J)) \leq \Def(I)+\Def(J).
\end{equation}

\begin{Lemma}\label{6}
Let $K= I\cap J$, then $\Def(I\cup J)\leq def (I) +def (J) -def (K)$.
\end{Lemma}

\begin{proof}
By the definition of the defect we have:
$$
\Def(I\cup J)= \dim(L(I\cup J)) -\#(I\cup J)=\dim(L(I\cup J)) -\#I-\#J+\#K,
$$
where $\#I,\#J,\#K$ are the sizes of $I,J,K$ respectively. But also 
$$
def (I) +def (J) -def (K) = \dim(L(I))+\dim(L(J))-\dim(L(K))-\#I-\#J+\#K,
$$
so we need to compare $\dim(L(I\cup J))$ and  $\dim(L(I))+\dim(L(J))-\dim(L(K))$. For this notice that 
$$
\dim(L(I\cup J))=\dim(L(I))+\dim(L(J))-\dim(L(I)\cap L(J)),
$$
and since $K\subset I\cap J$, the space $L(K)$ is a subspace of $L(I)\cap L(J)$, so 
$$
\dim(L(I\cup J))\leq \dim(L(I))+\dim(L(J))-\dim(L(K)),
$$
which finishes the proof.
\end{proof}

\begin{Corollary}
Let $J$ and $I$ be two not equal minimal by inclusion subcollections with minimal defect. Then $I\cap J=\varnothing$.
\end{Corollary}

\begin{proof}

Indeed, let $I\cap J=K\ne \varnothing$. Since $K\subset J$ and $K\ne J$, the defect of $K$ is larger than the defect of $J$, so $def (J) -def (K)<0$. But by Lemma \ref{6}
$$
\Def(I\cup J)\leq def (I) +def (J) -def (K)<\Def(I)= \Def(J),
$$
which contradicts $\Def(I)= \Def(J)=d(A)$.
\end{proof}

\begin{Lemma}\label{unique}
Let $A$ be a collection of finite subsets of  $\Z^n$ with $d(A)\leq0$, then the minimal by inclusion subcollection with minimal defect exists and is unique.
\end{Lemma}

\begin{proof}
 In the case $d(A) = 0$ the unique essential subcollection is the empty collection $J=\varnothing$.
 
For $d(A)<0$, existence is clear. For uniqueness, assume that $I$ and $J$ are two different minimal by inclusion subcollections with minimal defect, then by Lemma 1 $I\cap J = \varnothing$. But for disjoint subcollections $I, J$ by relation (\ref{sub}) we have:
$$
\Def(I\cup J) \leq \Def(I)+\Def(J) < \Def(I)=\Def(J),
$$
since $\Def(I)=\Def(J)=d(A)<0$. But this contradicts the minimality of $I$ and $J$.
\end{proof}

\subsection{Some properties of the essential subcollection.}

Let $A=(\cA_1,\ldots,\cA_k)$ be a collection of finite subsets of the lattice  $\Z^n$. For the subcollection $J$ denote by $J^c=\{1,\ldots,k\} \setminus J$ the compliment subcollection and by $\pi_J:\R^n\to \R^n/L(J)$ the natural projection. 

\begin{Lemma}\label{compdef}
In the notations above let $\pi_J(J^c)$ be the collection $(\pi_J(\cA_i))_{i\in J^c}$. Then the following relations hold: \\
1. $\Def(A)= \Def(J) + \Def(\pi_J(J^c)),$\\
2. $d(A)\geq d(J) + d(\pi_J(J^c)),$\\
3. if $J$ furthermore is the unique essential subcollection of $A$, then $d(\pi_J(J^c))=0$. 
\end{Lemma}

\begin{proof}
The proof of the part 1. is a direct calculation:
\[
\Def(J\cup J^c)=\dim (L(J\cup J^c)) - \#(J\cup J^c) = 
\] 
\[
\dim(L(J))+ \dim L(\pi_J(J^c))- \#(J)-\#(J^c) = \Def(J) + \Def(\pi_J(J^c)).
\]

For the part 2. note, that for any $B\subset J$, $C\subset J^c$ one has $L(B)\subset L(J)$ and hence the following is true:
\[
\Def(\pi_B(C)) \geq \Def(\pi_J(C)).
\]
This implies that:
\[
\Def(B\cup C)=\Def(B) + \Def(\pi_B(C))\geq \Def(B) + \Def(\pi_J(C)) \geq d(J) + d(\pi_J(J^c)).
\]

For the part 3. assume that $\Def(\pi_J(I)) < 0$ for some $I\subset J^c$. Then by part 1. we have
\[
\Def(J\cup I) = \Def(J) + \Def(\pi_J(I))<\Def(J).
\]
Since the defect of empty collection is 0, the minimal defect $d(\pi_J(I))$ is also 0.
\end{proof}

\begin{Proposition}\label{d+1}
Let $A=(\cA_1,\ldots,\cA_{k})$ be a collection of finite subsets of $\Z^n$ such that $\Def(A)=d(A)<0$. Let $J$ be the unique essential subcollection of the collection $A$. Then for any $i\in J$, the following is true: 
\[
\Def(A\setminus \{i\})=d(A\setminus \{i\})= d(A)+1.
\] 
\end{Proposition} 

\begin{proof}
For a collection $B$ and an element $b\in B$  the defect can not increase by more then 1 after removing $b$:
\[
\Def(B\setminus \{b\})\leq \Def(B)+1,
\] 
where the equality holds if and only if $L(B\setminus \{b\})=L(B)$. For the essential subcollection $J$ and any $i\in J$, the defect $\Def(J\setminus \{i\})$ is strictly greater then $\Def(J)$, so it is equal to $\Def(J)+1=\Def(A)+1$. Hence, $L(J)= L(J\setminus \{i\})$, and in particular $\pi_J=\pi_{J\setminus \{i\}}$.

Since $\Def(A)=d(A)=\Def(J)$, the defect  $\Def(\pi_{J}J^c)$ is equal to zero by part 1. of Lemma~\ref{compdef}. Moreover, one has:
\[
\Def(A\setminus \{i\})=\Def(J\setminus \{i\}) + \Def(\pi_{J \setminus \{i\}}J^c)= \Def(J\setminus \{i\}) + \Def(\pi_{J}J^c)= \Def(A)+1.
\] 

By part 2. and part 3. of Lemma~\ref{compdef} one has:
\[
d(A\setminus \{i\})\geq d(J\setminus \{i\}) + d(\pi_{J \setminus \{i\}}J^c)= d(J\setminus \{i\}) + d(\pi_{J}J^c)=d(J) + 1. 
\]
But since $\Def(A\setminus \{i\})=\Def(A)+1$, the minimal defect $d(A\setminus \{i\})$ is also equal to $\Def(A)+1$. 
\end{proof}

\begin{Corollary}\label{nconsist}
Let $A=(\cA_1,\ldots,\cA_{n+d})$ be a collection of finite subsets of $\Z^n$ such that $A$ is an essential collection of defect $-d$, i.e.
$$
-d = d(A) = \Def(A) < \Def(J),
$$
for any proper $J\subset \{1,\ldots, n+d\}$. Then there exists a subcollection $I$ of size $dim(L(A))=n$ with $d(I)=0$. 
\end{Corollary}

\begin{proof}
Apply Proposition~\ref{d+1} successively.
\end{proof}

\section{The main theorem.}

In this section we will prove the main theorem. For a collection $A=(\cA_1,\ldots, \cA_k)$ of finite subsets of $\Z^n$ and subcollection $J$ let $\cA_J, L(J)$, and $\pi_J$ be as before. For the subgroup $G$ of $\Z^n$ we will denote by $\ker(G)$ the set of points $p\in (\C^*)^n$ such that $g(p)=1$ for any $g\in G$. Furthermore, denote by
\begin{itemize}
\item $\Lambda(J) = L(J)\cap\Z^n$ the lattice of integral points in $L(J)$;
\item $G(J)$ the group generated by all the differences of the form $(a-b)$ with $a,b\in \cA_i$ for any $i\in J$;
\item $ind(J)$ the index of $G(J)$ in $\Lambda(J)$;
\item $\ker(G)$ the set of points $p\in (\C^*)^n$ such that $g(p)=1$ for any  $p\in G(J)$.
\end{itemize}

\subsection{Independence properties of systems} In this subsection we will prove independence theorems for the roots of generically consistent systems.

\begin{Lemma}\label{codim2}
Let $\cA\subset \Z^n$ be a finite subset of size at least 2 and let $p,q\in(\C)^*$ be such that $p/q \notin \ker(G(\cA))$. Then the set of Laurent polynomials $f$ with support in $\cA$, such that $f(p)=f(q)=0$ has codimension 2 in $\Omega_\cA$.
\end{Lemma}

\begin{proof}
Vanishing of $f$ at points $p$ and $q$ gives two linear conditions on the coefficients of $f$:
$$
\sum_{k\in A} a_k p^k =0,\quad \sum_{k\in A} a_k q^k =0.
$$
The relations above are independent unless $(p/q)^k=\lambda$ for some $\lambda$, and any $k\in \cA$. The later implies that $(p/q)^{k_1-k_2} = 1$ for any $k_1,k_2\in \cA$, i.e. $p/q \in \ker(G(\cA))$.
\end{proof}

\begin{Definition}
Let $T$ be an algebraic subgroup of $(\C^*)^n$ and $A=(\cA_1,\ldots \cA_n)$ be a collection  of finite subsets of $\Z^n$ such that $d(A)=0$. We would say that $A$ is $T$-independent if the generic system of Laurent polynomials $f_1,\ldots,f_n$ with supports in $A$ does not have two different roots $p,q\in (\C^*)^n$ with $p/q\in T$.
\end{Definition}

\begin{Corollary}\label{finite}
Let $A=(\cA_1,\ldots \cA_n)$ be a collection of finite subsets of $\Z^n$ and $G\subset (\C^*)^n$ be a finite subgroup such that $G\cap \ker(\cA_1,\ldots \cA_n)=1$, then the collection  is $G$-independent.
\end{Corollary}

\begin{proof}
Indeed, since $G\cap \ker(G(A))=1$, for each $g\in G$ there exist $i$ such that $g\not\in \ker(\cA_i)$. So the space of systems which vanish at a pair of different points $p$ and $q$ with $p/q\in G$ is a finite union of codimension at least 1 subspaces, which finishes the proof. 
\end{proof}

For an algebraic subgroup $T$ of $(\C^*)^n$ let $Lie(T)$ be its Lie algebra and $L_T\subset \R^n$ be its annihilator in the space of characters. In other words, $L_T$ is a linear span of the set of monomials which have value 1 on the identity component of a group $T$.

\begin{Theorem}\label{Tindepen}
Let $A=(\cA_1,\ldots \cA_n)$ be a collection of finite subsets of $\Z^n$ such that $d(A)=0$. Let $T$ be an algebraic subgroup of $(\C^*)^n$ such that $\ker G(A) \cap T = 1$ and for any subcollection $J$ such that $L_T\subset L_J$ the defect of $J$ is positive. Then the collection $A$ is $T$-independent.
\end{Theorem}
\begin{proof}
Let $k$ be the dimension of $T$, then $\dim L_T=n-k$. Since for any $J$ with $L_T\subset L_J$ the defect of $J$ is positive and $d(A)=0$, there are at most $n-k-1$ supports $\cA_i$ such that $L_i\subset L_T$.  Indeed, assume there is a subcollection $J$ of size $n-k$ with $L_J\subset L_T$, then $\Def(J)=0$ and $L_J=L_T$, which contradicts the assumptions. Therefore, there are at least $k+1$ supports $\cA_i$, say for  $i=1,\ldots, k+1$, with $\dim (T\cap \ker\cA_i) < k$. 

Define $T_1$ to be the union $\bigcup_{i=1}^{k+1} (T\cap \ker\cA_i)$ and $T' = T\setminus T_1$ to be its compliment. By Lemma~\ref{codim2}, the codimension of the set of systems $\bold f$ with supports in $A$ having roots $x$ and $px$ for the fixed $x\in (\C^*)^n$ and $p \in T'$ is at least $n+k+1$. Hence, the space of systems with supports in $A$  having 
 two different roots $p,q$ with $p/q\in T'$ has codimension at least 1.
 
 If the dimension of $T_1$ is positive, notice that $L_T \subset L_{T_1}$, and, therefore, for any $J$ such that $L_{T_1} \subset L_J$ the defect of $J$ is positive.  Hence, we can apply the above argument to $T_1$, and continue inductively until we obtain $T_l$ of dimension 0 (with  $T'_{l-1}=T_{l-1}\setminus T_l$).
 
Since $\dim T_l = 0$, i.e. $T_l$ is a finite subgroup of $(\C^*)^n$, by Corollary~\ref{finite} the space of systems with  two different roots $p,q$ with $p/q\in T_1$ has codimension at least 1.

In this manner we obtained the decomposition of $T$ in the finite disjoint union of subsets $\amalg_{i=0}^l T_i'$ (where $T_0'=T'$ and $T_l'=T_l$) such that for any $i$ the space of systems with  a pair of different roots $p,q$ with $p/q\in T_i$ has codimension at least 1. Therefore, the space of system with a pair of different roots $p$ and $q$ with $p/q\in T$ is a finite union of codimension at least 1 subspaces, and the theorem is proved.
\end{proof}

\begin{Corollary}
Let $\chi:(\C^*)^n\to \C^*$ be any character and $A=(\cA_1,\ldots \cA_n)$ be a collection of finite subsets of $\Z^n$ such that $G(A)=\Z^n$ and $\Def(J)>0$  for any proper nonempty subcollection $J$. Then the generic system of Laurent polynomials with supports in $A$ does not have a pair of different roots $p,q\in (\C^*)^n$ with $\chi(p)=\chi(q)$.
\end{Corollary}
\begin{proof}
Indeed, $\chi(p)=\chi(q)$ if and only if $p/q\in \ker(\chi)$, but the collection $A$ is $\ker(\chi)$-independent since it satisfies  the assumptions of Theorem~\ref{Tindepen} for any algebraic subgroup of~$(\C^*)^n$. 
\end{proof}

\subsection{Zero set of the generic essential system.}
In this subsection we will work with the systems of Laurent polynomials $f_1=\ldots=f_k=0$ with supports in $A=(\cA_1,\ldots,\cA_k)$ such that the essential subcollection is $A$ itself. We call such systems essential.

\begin{Theorem}\label{singpo}
Let $A=(\cA_1,\ldots,\cA_{n+d})$ be a collection of finite subsets of $\Z^n$ such that $ind(A)=1$. Let also $A$ be an essential collection, i.e.
$$
-d=d(A)=\Def(A)<\Def(J),
$$
for any proper $J\subset \{1,\ldots, n+d\}$. Then for a generic consistent system  $\bold f =(f_1,\ldots,f_k)\in X_A \subset \Omega_A$, the corresponding zero set $Y_{\bold f}$ is a single point.
\end{Theorem}

Here, and everywhere in this paper, by a generic point in algebraic variety $X$ parametrizing systems of Laurent polynomials we mean a point in $X\setminus \Sigma$ for a fixed subvariety $\Sigma$ of smaller dimension.

\begin{proof}
By Proposition~\ref{nconsist} there exists a subcollection $I$ of $A$ of size $n$ with $d(I)=0$. Without loss of generality let us assume that $I=\{1,\ldots,n\}$. The space $\Omega_A$ of polynomials with supports in $A$ could be thought as a product
$$
\Omega_A=\Omega_I\times \Omega_{I^c},
$$ 
where $\Omega_I$ and $\Omega_{I^c}$ are the spaces of systems of Laurent polynomials with supports in $I$ and $I^c$ respectively. Let $p: \Omega_A \to \Omega_I$ be the natural projection on the first factor.

By the Bernstein criterion the subsystem $f_1=\ldots=f_n=0$ is generically consistent. Moreover, the BKK Theorem asserts that the generic number of solutions in $(\C^*)^n$ is $n!Vol(\Delta_1,\ldots,\Delta_n)$, where $\Delta_i$ is the convex hull of $\cA_i$, and in particular is finite. Let us denote by $\Omega^{gen}_I\subset \Omega_I$ the Zariski open subset of systems $f_1=\ldots=f_n=0$ with exactly $n!Vol(\Delta_1,\ldots,\Delta_n)$ roots.

For each point $\bold f_I \in \Omega^{gen}_I$ the preimage $p^{-1}(\bold f_I)$ of the projection $p$ restricted to the set of consistency $X_A$ is a union of $n!Vol(\Delta_1,\ldots,\Delta_n)$ vector spaces $V_j(\bold f_I)$'s of dimension $|\cA_{n+1}|+\ldots+|\cA_{n+d}|-d$ each. The intersection of any two of these vector spaces has smaller dimension for generic $\bold f_I \in \Omega^{gen}_I$. Indeed, since $G(A)=\Z^n$ and $A$ is essential, the ussumptions of Theorem~\ref{Tindepen} are satisfied for the collection $I$ and subgroup $\ker(I^c)$ of  $(\C^*)^n$. Hence $I$ is $\ker(I^c)$-independent by Theorem~\ref{Tindepen}.

 Denote by $X'_A\subset X_A$  the set of points which belongs to exactly one of the $V_j(\bold f_I)$'s. By construction, the dimension of $X'_A$ is equal to $|\cA_{1}|+\ldots+|\cA_{n+d}|-d = \dim(X_A)$. Since $X_A$ is irreducible, the  complement $\Sigma=X_A\setminus X'_A$ is an algebraic subvariety of smaller dimension. But for any $\bold f\in X'_A$ the zero set $Y_{\bold f}$ is a single point, so the theorem is proved.
\end{proof}

\begin{Corollary}\label{ess}
Let $A=(\cA_1,\ldots,\cA_{k})$ be an essential collection of finite subsets of $\Z^n$ of defect $d(A)=\Def(A)=-d$. Then for the generic  $\bold f\in X_A \subset \Omega_A$ the zero set $Y_{\bold f}$ is a finite disjoint union of $ind(A)$ subtori of dimension $n-k+d$ which are different by a multiplication by elements of $(\C^*)^n$.
\end{Corollary}

\begin{proof}
The lattice $G(A)$ generated by all of the differences in $\cA_i$'s defines a torus $T\simeq (\C^*)^{k-d}$ for which $G(A)$ is the lattice of characters. The inclusion $G(A)\hookrightarrow \Z^n$ defines the homomorphism:
$$
p: (\C^*)^n \to T.
$$
The kernel of the homomorphism $p$ is the subgroup of $(\C^*)^n$ consisting of finite disjoint union of $ind(A)$ subtori of dimension $n-k+d$ which are different by a multiplication by elements of $(\C^*)^n$.

The multiplication of Laurent polynomials by monomials does not change the zero set of a system. For any $i$ let $\widetilde \cA_i$  be any translation of $\cA_i$ belonging to $G(J)$. We can think of $\widetilde \cA_i$ as support of a Laurent polynomial on $T$. We will denote by $\widetilde A$ the collection $(\widetilde \cA_1,\ldots,\widetilde \cA_{k})$ understood as a collection of supports of Laurent polynomials on the torus $T$. The collection $\widetilde A$ satisfies the assumptions of Theorem~\ref{singpo}. 

With a system $\bold f\in \Omega_A$ one can associate a system of Laurent polynomials $  \bold{\widetilde f}$ on $T$ in a way described above. The zero set of $Y_\bold f$ of a system $\bold f$ is given by
$$
Y_\bold f = p^{-1} (Y_{\bold{\widetilde f}})\,(\textnormal{in particular}\,Y_\bold f \simeq Y_{\bold{\widetilde f}}\times ker(p)),
$$
where $Y_{ \bold{\widetilde f}}$ is the zero set of the system $ \bold{\widetilde f}$ on $T$. By Theorem \ref{singpo} for the generic system $ \bold{\widetilde f} \in X_{\widetilde A} \subset \Omega_{\widetilde A}$ the zero set $Y_{ \bold{\widetilde f}}$ which finishes the proof.
\end{proof}

\subsection{General systems}

\begin{Theorem}\label{main}
Let $A=(\cA_1,\ldots,\cA_{k})$ be a collection of finite subsets of $\Z^n$ with the essential subcollection $J$. Then for the generic system $\bold f\in X_A\subset \Omega_A$ the zero set $Y_\bold f$ is a disjoint union of $ind(J)$  varieties $Y_1,\ldots, Y_{ind(J)}$ each of which is given by a $\Delta$-nondegenerate system with the same Newton polyhedra.
\end{Theorem}

Theorem~\ref{main} provides a solution for the problem of computing discrete invariants of the zero set of generic consistent system with generically inconsistent supports by reducing it to the classical Newton polyhedra theory. The concrete examples of applications of Theorem~\ref{main} are given in the next section.

\begin{proof}
Without loss of generality let us assume that $J=\{1,\ldots, l\}$. By Corollary \ref{ess} there exists a Zariski open subset  $X'_A\subset X_A$, such that for any $\bold f =(f_1,\ldots,f_k)\in X'_A$ the zero set of the system $f_1=\ldots=f_l=0$ is a finite disjoint union of $ind(J)$ subtori $V_1,\ldots,V_{ind(J)}$ which are different by a multiplication by an element of $(\C^*)^n$.

For the generic point  $\bold f =(f_1,\ldots,f_k)\in X'_A$ the restrictions of Laurent polynomials $f_{l+1},\ldots,f_k$ to each $V_i$ are non-degenerate Laurent polynomials with Newton polyhedra $\pi_J(\Delta_{l+1}),\ldots, \pi_J(\Delta_{k})$, respectively.
\end{proof}

\begin{Corollary}
For the generic system $\bold f\in X_A\subset \Omega_A$ the zero set $Y_\bold f$ is a non-degenerate complete intersection. That is $Y_\bold f$ is defined by $codim(Y_\bold f)$ equations with independent differentials.
\end{Corollary}

\begin{proof}
Indeed, each of the components $Y_i\subset V_i$ of $Y_\bold f$ is defined by the  restrictions of Laurent polynomials $f_{l+1},\ldots,f_k$ to $V_i$, and hence is a non-degenerate complete intersection in $V_i$ for generic consistent system $\bold f$.  

But the union of shifted subtori $V_1,\ldots,V_{ind(J)}$ could be defined by the $codim (V_i)$ more independent equations in $(\C^*)^n$, which finishes the proof.
\end{proof}

\section{Discrete invariants}

Theorem~\ref{main} asserts that any discrete invariant which can be computed by means of the theory of Newton polyhedra could be also computed  for the zero set $Y_\bold f$ of generic consistent system with generically inconsistent supports. In this section we will give two examples of such calculations, but absolutely the same strategy is applicable to other discrete invariants such as Hodge--Deligne numbers, or the number of connected components (which were computed in the classical case in \cite{DKh} and \cite{Kh2} respectively).

Through all of this section by the volume on a vector space $V$ with a lattice $\Lambda$ inside we mean  the translation invariant volume normalized by the following condition: for any  $v_1,\ldots,v_{k}$  which are generators of the lattice $\Lambda$, the volume of the parallelepiped with sides $v_1,\ldots,v_{k}$ is equal to 1.

\begin{Theorem}[Number of roots]
Let $\cA_1,\ldots,\cA_{n+k}\subset \Z^n$ be such that $d(\cA)=-k$ and $J$ be the unique essential subcollection. Then the zero set $Y_\bold f$ of the generic consistent system has dimension 0, and the number of points in $Y_\bold f$ is equal to
 $$
 (n-\#J + k)! \cdot ind(J) \cdot Vol(\pi_J(\Delta_i)_{i\notin J}),
 $$
 where $\Delta_i$ is the convex hull of $\cA_i$ and $Vol$ is the mixed volume on $\R^n/L(J)$ normalized with respect to the lattice $\Z^n/\Lambda(J)$.
\end{Theorem}
If $k=0$ this theorem coincides with the BKK theorem. In the case $k=1$ the generic number of solution appears as the corresponding degree of $A$-resultant and was computed in \cite{D'AS}.

\begin{proof}
First note that for generic $\bold f \in X_\cA$ the dimension $\dim(Y_\bold f)$ is equal to $\dim(\widetilde X_\cA)-\dim(X_\cA)=0$. By Theorem~\ref{main} the generic zero set $Y_\bold f$ is a disjoint union of $ind(J)$ varieties $Y_1,\ldots, Y_{ind(J)}$ each of which is defined by generic system with Newton polyhedra  $\pi_J(\Delta_i)$ for $i\notin J$. By the BKK formula the number of points in $Y_i$ is finite and is equal to $(n-\#J + k)! Vol(\pi_J(\Delta_i)_{i\notin J}))$. Therefore, the number of points in $Y_\bold f$ is
$$
|Y_\bold f| = \sum_{i=1}^{ind(J)} |Y_i|= (n-\#J + k)! \cdot ind(J) \cdot Vol(\pi_J(\Delta_i)_{i\notin J}).
$$
\end{proof}

For simplicity, we will formulate next theorem in the ``hypersurface'' case, i.e. when the essential subcollection contains all but one supports (the general case could be deduced similarly). 
\begin{Theorem}
Let $\cA_1,\ldots,\cA_{k}\subset \Z^n$ be such that $d(\cA)<0$ and let $J=\{2,\ldots,k\}$ be the unique essential subcollection. Then the Euler characteristic and the geometric genus of the zero set $Y_\bold f$ of the generic consistent system is given by
 $$
 \chi(Y_\bold f) = (-1)^{n-dim(J)-1}(n-dim(J))! \cdot ind(J) Vol(\pi_J(\Delta_1)),
 $$
  $$
 g(Y_\bold f)=ind(J)\left( B^+(\pi_J(\Delta_1))\right),
 $$
 where $\Delta_1$ is the convex hull of $\cA_1$, $Vol$ is the volume on $\R^n/L(J)$ normalized with respect to the lattice $\Z^n/\Lambda(J)$, and $B^+(\Delta)$ is the number of integral point in the interior of $\Delta$.
\end{Theorem}

\begin{proof}
Indeed, by Theorem~\ref{main} the generic zero set $Y_\bold f$ is a disjoint union of $ind(J)$ varieties $Y_1,\ldots, Y_{ind(J)}$ each of which is defined by a generic equation with a Newton polyhedra $\pi_J(\Delta_1)$. Therefore, the Euler characteristic of $Y_i$ is given by $\chi(Y_i)= (-1)^{n-dim(J)-1}(n-dim(J))!\cdot Vol(\pi_J(\Delta_1))$ (see \cite{Kh1}), and the geometric genus of $Y_i$ is given by $g(Y_i)=B^+(\pi_J(\Delta_1))$ (see \cite{Kh}). The theorem follows from the additivity of the Euler characteristic and the geometric genus.
\end{proof}


\begin{thebibliography}{[D'AS]}
\nopagebreak[3]


\bibitem[B]{B}
Bernstein, David N. The number of roots of a system of equations. Functional Analysis and its applications 9.3, (1975) 183-185.

\bibitem[D'AS]{D'AS}
D'Andrea, C., and Sombra, M.  A Poisson formula for the sparse resultant. Proceedings of the London Mathematical Society (2015).

\bibitem[DKh]{DKh}
Danilov, V. I.,  Khovanskii, A. G.  Newton polyhedra and an algorithm for computing Hodge -- Deligne numbers. Mathematics of the USSR-Izvestiya, 29(2), (1987).


\bibitem[GKZ]{GKZ}
Gelfand, I. M., Kapranov, M.,  Zelevinsky, A.  Discriminants, resultants, and multidimensional determinants. Springer Science \& Business Media (2008).


\bibitem[Kh]{Kh}
Khovanskii, A. G. Newton polyhedra and the genus of complete intersections. Functional Analysis and its applications, 12(1), (1978) 38-46.

\bibitem[Kh1]{Kh1}
Khovanskii, A. G.  Algebra and mixed volumes. In book Y.D. Burago and V.A. Zalgaller, Geometric inequalities, Springer-Verlag, Berlin and New York. V. 285, (1988) 182Ð207.


\bibitem[Kh2]{Kh2}
Khovanskii, A. G.  Newton polytopes and irreducible components of complete intersections. Izvestiya: Mathematics, 80(1), (2016).


\bibitem[St]{St}
Sturmfels, B.  On the Newton polytope of the resultant. Journal of Algebraic Combinatorics, 3(2), (1994) 207-236.

\end{thebibliography}
\end{document}